\documentclass[11pt]{amsart}

\usepackage[T1]{fontenc}
\usepackage[utf8]{inputenc}
\usepackage{lmodern}
\usepackage{amsmath,amssymb,amsthm,mathtools}
\usepackage{booktabs}
\usepackage{enumitem}
\usepackage[expansion=false]{microtype}
\usepackage[margin=1.08in]{geometry}
\usepackage[hidelinks]{hyperref}
\hypersetup{
  pdftitle={Gross vectors modulo 2 and elliptic curves of prime conductor},
  pdfauthor={Matija Kazalicki; Sini\v sa Slijep\v cevi\'c},
  pdfsubject={Gross vectors, supersingular elliptic curves, and modular degrees},
  pdfkeywords={supersingular elliptic curves, Brandt modules, Gross vectors, binary quadratic forms, modular degree, Watkins' conjecture}
}

\newtheorem{theorem}{Theorem}[section]
\newtheorem{proposition}[theorem]{Proposition}
\newtheorem{lemma}[theorem]{Lemma}
\newtheorem{corollary}[theorem]{Corollary}
\theoremstyle{definition}

\newtheorem*{definition*}{Definition}

\theoremstyle{remark}
\newtheorem{remark}[theorem]{Remark}

\newcommand{\Q}{\mathbb{Q}}
\newcommand{\Z}{\mathbb{Z}}
\newcommand{\F}{\mathbb{F}}
\newcommand{\OO}{\mathcal{O}}
\newcommand{\Pic}{\operatorname{Pic}}
\newcommand{\End}{\operatorname{End}}
\newcommand{\Gal}{\operatorname{Gal}}
\newcommand{\tr}{\operatorname{tr}}
\newcommand{\Nr}{\operatorname{Nr}}
\newcommand{\Span}{\operatorname{span}}
\newcommand{\leg}[2]{\left(\frac{#1}{#2}\right)}
\newcommand{\crow}{\mathbf{c}}
\newcommand{\Rp}{\mathcal{R}_p}

\title[Gross vectors modulo $2$]{Gross vectors modulo $2$ and elliptic curves of prime conductor}
\author{Matija Kazalicki}
\address{Department of Mathematics, Faculty of Science, University of Zagreb\\
Bijeni\v{c}ka cesta 30, 10000 Zagreb, Croatia}
\email{matija.kazalicki@math.hr}

\author{Sini\v sa Slijep\v cevi\'c}
\address{Department of Mathematics, Faculty of Science, University of Zagreb\\
Bijeni\v{c}ka cesta 30, 10000 Zagreb, Croatia}
\email{slijepce@math.hr}

\date{}


\begin{document}

\begin{abstract}
Let $p>3$ be a prime, and let $S_p$ denote the geometric isomorphism
classes of supersingular elliptic curves in characteristic $p$ whose
$j$-invariants lie in $\F_p$. For each negative fundamental discriminant
$-D$ for which $p$ is inert in $\Q(\sqrt{-D})$, let $m_i(D)$,
$i\in S_p$, be the integral coefficients of the corresponding Gross
vector. We prove that the vectors
\[
  \bigl(m_i(D)\bmod2\bigr)_{i\in S_p}
\]
span $\F_2^{S_p}$. The key step reduces the parity of the representation
numbers of Gross's ternary lattices to representation by rank-two
sublattices perpendicular to Frobenius. Using Ibukiyama's explicit
maximal orders, these binary forms are identified with those occurring
in the Xiao--Zhou--Deng--Qu parametrization of supersingular elliptic
curves over $\F_p$. Class field theory and Chebotarev's theorem then
allow the individual supersingular coordinates to be isolated.

As a consequence, if $E/\Q$ has prime conductor $p$ and positive
Mordell--Weil rank, then every coefficient of its Brandt eigenvector
indexed by $S_p$ is even, proving a conjecture of Kazalicki and Kohen.
Thus an odd coefficient at a rational supersingular class is an algebraic
certificate of rank $0$. Combining this parity theorem with formulas of
Mestre and Gross--Kudla, we also prove that the modular degree of every
positive-rank elliptic curve of prime conductor and root number $+1$ is
divisible by $4$. Consequently, Watkins' conjecture holds for all such
curves of rank $2$.
\end{abstract}

\maketitle
\enlargethispage{8pt}

\section{Introduction}

Let $p>3$ be a prime, and let $S$ denote the set of isomorphism classes of supersingular elliptic curves over $\overline{\F}_p$. The Brandt module $X=\bigoplus_{i\in S}\Z e_i$ carries the usual Hecke operators. Let $E/\Q$ be an elliptic curve of conductor $p$, and let $f_E=\sum_{n\ge1}a_n(E)q^n$ be its normalized newform. By the definite Jacquet--Langlands correspondence, $f_E$ determines a one-dimensional simultaneous Hecke eigenspace in $X\otimes\mathbb C$ on which, for every prime $\ell\ne p$, the Hecke operator $t_\ell$ acts by multiplication by $a_\ell(E)$. This eigenspace contains a primitive vector with integer coefficients, unique up to sign, which we write as
\[
  v_E=\sum_{i\in S}c_i e_i;
\]
see, for example, \cite[Section~5]{Gross1987}.  We write
\[
  S_p=\{i\in S:j(E_i)\in\F_p\}
\]
for the supersingular classes fixed by Frobenius. We prove that positive Mordell--Weil rank forces $c_i$ to be even for every $i\in S_p$. Equivalently, a single odd coefficient at a rational supersingular class certifies that $\operatorname{rank}E(\Q)=0$.

For $i\in S$, let $R_i=\End_{\overline{\F}_p}(E_i)$ and $w_i=\#R_i^\times/2$. By Deuring's correspondence, $R_i$ is a maximal order in the quaternion algebra $B_{p,\infty}$ ramified at $p$ and $\infty$ \cite{Deuring1941}. We use the Brandt pairing $\langle e_i,e_j\rangle=w_i\delta_{ij}$. Let $-D$ be a negative fundamental discriminant satisfying $\leg{-D}{p}=-1$. If $h_i(-D)$ denotes the number of optimal embeddings $\OO_{-D}\hookrightarrow R_i$, modulo $R_i^\times$-conjugacy, and $u(-D)=\#\OO_{-D}^{\times}$, Gross \cite[Sections~12--13]{Gross1987} defines the CM vector
\[
  b_D=\sum_{i\in S}\frac{h_i(-D)}{u(-D)}e_i.
\]
We use the integral coefficients
\[
  m_i(D)=\langle e_i,b_D\rangle=\frac{w_i h_i(-D)}{u(-D)}.
\]

\begin{definition*}
A negative fundamental discriminant $-D$ is \emph{admissible} for $p$ if $\leg{-D}{p}=-1$, equivalently if $p$ is inert in $\Q(\sqrt{-D})$. For such $-D$, the \emph{Gross parity row} is
\[
  \crow(D)=\bigl(m_i(D)\bmod2\bigr)_{i\in S_p}\in\F_2^{S_p},
\]
and the \emph{Gross-row space} is
\[
  \Rp=\Span_{\F_2}\{\crow(D):-D\text{ admissible}\}.
\]
\end{definition*}

Our main result is the following.

\begin{theorem}[Gross-row spanning]\label{thm:spanning}
For every prime $p>3$,
\[
  \Rp=\F_2^{S_p}.
\]
\end{theorem}

Suppose now that $E(\Q)$ has positive rank. By Kolyvagin--Logachev \cite[Theorem~0.3]{KolyvaginLogachev1990}, $L(E,1)=0$, and Gross's central-value formula \cite[Proposition~13.5]{Gross1987} gives $\langle v_E,b_D\rangle=0$ for every admissible $-D$; see Proposition~\ref{prop:gross-orthogonality}. If $\bar i$ denotes the Frobenius conjugate of $i$, then $m_{\bar i}(D)=m_i(D)$, while Lemma~\ref{lem:frob-sign} gives $c_{\bar i}=\varepsilon(E)c_i$, where $\varepsilon(E)$ is the root number. In particular, $c_{\bar i}\equiv c_i\pmod2$. The contributions of the nonrational supersingular classes therefore cancel in Frobenius pairs modulo $2$, giving
\[
  \sum_{i\in S_p}c_i m_i(D)\equiv0\pmod2
\]
for every admissible $-D$. Hence $(c_i\bmod2)_{i\in S_p}\in\Rp^\perp$, and Theorem~\ref{thm:spanning} gives the following.

\begin{corollary}[Kazalicki--Kohen parity]\label{cor:parity}
If $E/\Q$ has prime conductor $p>3$ and positive Mordell--Weil rank, then
\[
  c_i\equiv0\pmod2\qquad(i\in S_p).
\]
\end{corollary}
Consequently, an odd coefficient $c_i$ for some $i\in S_p$ forces $\operatorname{rank}E(\Q)=0$. This is a one-sided criterion: the converse is false. For example, the rank-zero curve
\[
  E:\quad y^2+y=x^3+x^2-1873x-31833
\]
of conductor $37$ has positive discriminant and no rational point of order
$2$, so \cite[Theorem~1.4]{KazalickiKohen2017} shows that
$c_i\equiv0\pmod2$ for every $i\in S_{37}$.

This parity question arose from the study of divisor polynomials. Ono observed a connection between their zeros modulo $p$ and supersingular $j$-invariants \cite[p.~118]{Ono2004}. Kazalicki and Kohen made this connection precise for elliptic curves of prime conductor by relating the supersingular zeros of the divisor polynomial to the coefficients of the corresponding Brandt eigenvector \cite{KazalickiKohen2017}. When the root number is $-1$, they proved that $c_i=0$ for every $i\in S_p$. For root number $+1$, they conjectured that positive rank forces $c_i\equiv0\pmod2$ for every $i\in S_p$, and proved this when the minimal discriminant is positive and $E$ has no rational point of order $2$. Corollary~\ref{cor:parity} removes both restrictions.

The connection with Watkins' conjecture was developed in \cite{KazalickiKohen2018,KazalickiKohenCorrigendum2019}. Watkins conjectured that, for an elliptic curve $E/\Q$,
\[
  2^{\operatorname{rank}E(\Q)}\mid m_E,
\]
where $m_E$ denotes the modular degree \cite[Conjecture~4.1]{Watkins2002}. Kazalicki and Kohen showed how parity on $S_p$, together with Mestre's
norm formula \cite[Section~3, formula~(5), Theorem~3.2]{Mestre1986} and
the Gross--Kudla cubic identity
\cite[Corollary~11.5]{GrossKudla1992}, yields divisibility by $4$ of the
modular degree. Using their previously established parity theorem, this proved Watkins'
conjecture for rank-two elliptic curves of prime conductor and positive
discriminant. Corollary~\ref{cor:parity} removes the discriminant restriction and gives,
for every positive-rank prime-conductor curve with root number $+1$,
the divisibility statement proved below in
Theorem~\ref{thm:modular-degree}.

\begin{corollary}
Let $E/\Q$ be an elliptic curve of prime conductor $p>3$, positive Mordell--Weil rank, and root number $+1$. Then
\[
  4\mid m_E.
\]
In particular, Watkins' conjecture holds for such curves of rank $2$.
\end{corollary}

We briefly describe the proof of Theorem~\ref{thm:spanning}. Gross's
coefficient $m_i(D)$ is half the number of norm-$D$ vectors in the ternary
Gross lattice
\[
  L_i=\{x\in\mathbb Z+2R_i:\operatorname{tr}(x)=0\}.
\]
For $i\in S_p$, choose an $\F_p$-model of $E_i$ with Frobenius $\pi_i$ and
consider the rank-two lattice
\[
  L_i^\perp=L_i\cap(\Q\pi_i)^\perp.
\]
Conjugation by $\pi_i$, together with negation, groups the norm-$D$ vectors
outside $L_i^\perp$ into four-element orbits. It follows that, modulo $2$,
$m_i(D)$ is determined by representation by $L_i^\perp$. Thus Frobenius
reduces the parity problem from a ternary to a binary quadratic form.

To identify these binary forms, we use Ibukiyama's explicit maximal-order
models in $B_{p,\infty}$, keeping track of the quadratic order generated by
Frobenius. The resulting marked-order identification sends $\pi_i$ to the
distinguished quadratic generator in the corresponding model, and hence
identifies $L_i^\perp$ with its perpendicular lattice. A direct norm
calculation then gives either a primitive positive-definite binary quadratic
form of discriminant $-16p$, or four times a primitive form of discriminant
$-p$. We refer to these as the discriminant $-16p$ and discriminant $-p$
branches.

For each nonexceptional $i$, let $C_i$ denote the proper class of the
primitive form obtained in this way. Thus the natural orientation-free
object attached to $i$ is the inverse orbit
\[
  \{C_i,C_i^{-1}\},
\]
that is, the orbit of $C_i$ under the involution $C\mapsto C^{-1}$.
Comparing the formulas with the parametrization of Xiao--Zhou--Deng--Qu
\cite[Theorems~3.3, 3.5 and~3.7]{XZDQ2025} shows that, within either
branch, distinct geometric supersingular classes determine distinct
inverse orbits:
\[
  i\ne j
  \qquad\Longrightarrow\qquad
  \{C_i,C_i^{-1}\}\ne\{C_j,C_j^{-1}\}.
\]
In the discriminant $-16p$ branch, the classes $C_i$ also
lie in the nonprincipal genus.

The spanning theorem is then proved by applying Chebotarev's theorem in
the corresponding ring class fields. The exceptional coordinates $j=0$
and $j=1728$, when supersingular, are handled separately by explicit Gross
rows. For a nonexceptional coordinate $i$ in the discriminant $-16p$
branch, we choose a rational prime $\ell$ that splits in the relevant
quadratic order so that the two ideals above $\ell$ represent the classes
$C_i$ and $C_i^{-1}$. Since these are the only ideals of norm $\ell$, the
inverse orbit attached to $i$ has odd representation parity, while every
other nonexceptional inverse orbit has even parity. The genus conditions
ensure that $-\ell$ is admissible, and the relation between Gross
coefficients and binary representation numbers then shows that, modulo the
exceptional coordinates, the corresponding Gross row isolates $i$. Thus
all coordinates in the discriminant $-16p$ branch lie in $\mathcal R_p$.

When $p\equiv3\pmod4$, we next treat the discriminant $-p$ branch. For a
coordinate $i$ in this branch, Chebotarev is used with an additional
congruence condition to choose a prime $q\equiv1\pmod4$ whose two ideals
above $q$ represent $C_i$ and $C_i^{-1}$. The discriminant $-4q$ is then
admissible, and the corresponding Gross row isolates $i$ on the
discriminant $-p$ branch. It may also have support on the
discriminant $-16p$ branch, but those coordinates have already been shown
to lie in $\mathcal R_p$. Hence $e_i\in\mathcal R_p$ for every remaining
coordinate, proving
\[
  \mathcal R_p=\F_2^{S_p}.
\]

Products of two suitably chosen split primes give an additional support
phenomenon: in the discriminant $-16p$ branch, a Gross row associated with
such a product can realize any prescribed pair of nonexceptional
coordinates. We record this after the proof of
Theorem~\ref{thm:spanning}.

Section~\ref{sec:gross} fixes the Gross normalization and reduces the
elliptic-curve parity statement to Theorem~\ref{thm:spanning}.
Section~\ref{sec:binary} relates Gross coefficients modulo $2$ to binary
quadratic forms using Frobenius and identifies the resulting form classes.
Section~\ref{sec:classfield} develops the class-field-theoretic tools used
to control their representation parities, and
Section~\ref{sec:spanning-proof} proves the spanning theorem and records
the additional support relations arising from products of two primes.
Finally, Section~\ref{sec:watkins} gives the modular-degree application,
and Section~\ref{sec:example83} illustrates the argument for $p=83$.

\section{Gross relations and the \texorpdfstring{mod-$2$}{mod-2} reduction}\label{sec:gross}

We retain the notation of the introduction. Most of the ingredients in this
section are standard, and we record them in the precise normalization needed
below, where factors of $2$ are essential. The Frobenius symmetry of Brandt
eigenvectors is recalled from \cite[Propositions~18--19]{KazalickiKohen2017},
the interpretation of Gross coefficients as ternary representation numbers
comes from \cite[Section~12, in particular Proposition~12.9]{Gross1987},
and the exact orthogonality relation follows from Gross's central-value
formula. We then extract two elementary consequences needed for the proof:
modulo $2$ the Gross relations may be restricted to the Frobenius-fixed
supersingular classes, and the exceptional coordinates $j=0,1728$ are
supplied directly by the discriminants $-3$ and $-4$.

\subsection{Frobenius on Brandt eigenvectors}

Write $\bar i$ for the Frobenius conjugate of $i\in S$.

\begin{lemma}[Frobenius sign]\label{lem:frob-sign}
Let $f\in S_2(\Gamma_0(p))$ be a newform with root number $\varepsilon(f)$, and write its Brandt eigenvector as $v_f=\sum_i c_i e_i$.  Then
\[
  c_{\bar i}=\varepsilon(f)c_i
\]
for every $i$.  In particular, $c_{\bar i}\equiv c_i\pmod2$; if $\varepsilon(f)=+1$, then $c_{\bar i}=c_i$.
\end{lemma}

\begin{proof}
The degree-$p$ Brandt operator sends $e_i$ to $e_{\bar i}$.  Under Jacquet--Langlands it corresponds to $-W_p$.  Since the $W_p$-eigenvalue of a weight-two newform of prime level is $-\varepsilon(f)$, we have $t_pv_f=\varepsilon(f)v_f$.  Comparing coefficients gives the result.  This is the normalization used in \cite[Propositions~18--19]{KazalickiKohen2017}.
\end{proof}

\subsection{Gross coefficients as representation numbers}

For an admissible discriminant $-D$, let $b_D$ and $m_i(D)$ be as in the introduction.  We use throughout the full-unit convention $u(-D)=\#\OO_{-D}^{\times}$.  Define Gross's trace-zero lattice
\[
  L_i=\{x\in\Z+2R_i:\tr(x)=0\},
\]
and let $N_i(D)$ be the number of vectors $x\in L_i$ with $\Nr(x)=D$.

\begin{lemma}[Vector--embedding normalization]\label{lem:normalization}
For every negative fundamental discriminant $-D$,
\[
  N_i(D)=2m_i(D).
\]
In particular, $m_i(D)\in\Z$.
\end{lemma}

\begin{proof}
This is Gross's description of the coefficients of the ternary theta
series associated with $\Z+2R_i$; see
\cite[(12.7)--(12.8) and Proposition~12.9]{Gross1987}.
Indeed, Gross writes
\[
  \frac12\sum_{x\in L_i}q^{\Nr(x)}
  =\frac12+\sum_{D>0}a_i(D)q^D,
\]
so that $a_i(D)=N_i(D)/2$.  For $-D$ fundamental,
Proposition~12.9 gives
\[
  a_i(D)=\frac{w_i}{2}\frac{h_i(-D)}{u_{\rm Gross}(-D)}.
\]
Gross uses $u_{\rm Gross}(-D)=\#\OO_{-D}^{\times}/2$, whereas our
convention is $u(-D)=\#\OO_{-D}^{\times}$. Hence
\[
  a_i(D)=\frac{w_i h_i(-D)}{u(-D)}=m_i(D),
\]
which proves the assertion.
\end{proof}

\subsection{Gross orthogonality and the rational relation}

\begin{proposition}[Gross orthogonality]\label{prop:gross-orthogonality}
Let $E/\Q$ be an elliptic curve of prime conductor $p>3$ and positive Mordell--Weil rank, and let $v_E$ be its primitive integral Brandt eigenvector.  Then, for every admissible $-D$,
\[
  \langle v_E,b_D\rangle=0.
\]
\end{proposition}

\begin{proof}
Under the present fundamental-discriminant and inertness hypotheses, Gross's central-value formula gives
\[
  L(E,1)L(E\otimes\varepsilon_{-D},1)
  =C(E,D)\frac{\langle v_E,b_D\rangle^2}{\langle v_E,v_E\rangle},
  \qquad C(E,D)\ne0;
\]
see \cite[Proposition~13.5]{Gross1987} and \cite[Proposition~24]{KazalickiKohen2017}.  By Kolyvagin--Logachev, $L(E,1)\ne0$ implies that $E(\Q)$ is finite \cite[Theorem~0.3]{KolyvaginLogachev1990}.  Thus positive Mordell--Weil rank implies $L(E,1)=0$, and the formula forces $\langle v_E,b_D\rangle=0$.
\end{proof}

Frobenius transports optimal embeddings in $R_i$ bijectively to optimal embeddings in $R_{\bar i}$, and therefore
\[
  m_{\bar i}(D)=m_i(D).
\]
Combining this symmetry with Lemma~\ref{lem:frob-sign} turns the exact Gross relation into a relation supported only on $S_p$ modulo $2$.

\begin{corollary}[Rational Gross relation]\label{cor:rational-relation}
Under the hypotheses of Proposition~\ref{prop:gross-orthogonality},
\[
  \sum_{i\in S_p}c_i m_i(D)\equiv0\pmod2
\]
for every admissible $-D$.
\end{corollary}

\begin{proof}
Expanding the exact pairing gives $0=\sum_i c_i m_i(D)$.  If $i\notin S_p$, then $i\ne\bar i$ and the pair $\{i,\bar i\}$ contributes
\[
  c_i m_i(D)+c_{\bar i}m_{\bar i}(D)\equiv0\pmod2
\]
by Lemma~\ref{lem:frob-sign}.  Only the Frobenius-fixed classes remain.
\end{proof}

Two rational supersingular coordinates can be produced without class-field theory and will serve as exceptional coordinates in the prime and semiprime constructions.

\begin{lemma}[Exceptional rows]\label{lem:exceptional-rows}
If $j=0$ is supersingular, then $-3$ is admissible and
\[
  \crow(3)=e_0.
\]
If $j=1728$ is supersingular, then $-4$ is admissible and
\[
  \crow(4)=e_{1728}.
\]
For a positive-rank prime-conductor elliptic curve, the corresponding exact Gross relations give $c_0=0$ and $c_{1728}=0$ whenever these classes occur.
\end{lemma}

\begin{proof}
The class $j=0$ is supersingular exactly when $p\equiv2\pmod3$, and an
embedding of $\OO_{-3}$ into $R_i=\End(E_i)$ forces $E_i$ to have an
automorphism of order $3$, hence $j(E_i)=0$.  There are two such oriented
embeddings modulo $R_0^\times$-conjugacy, so
\[
  m_0(3)=\frac{3\cdot2}{6}=1.
\]
Thus $\crow(3)=e_0$.  Similarly, $j=1728$ is supersingular exactly when
$p\equiv3\pmod4$; an embedding of $\OO_{-4}$ occurs only at $j=1728$, and
the two oriented embeddings give
\[
  m_{1728}(4)=\frac{2\cdot2}{4}=1.
\]
Hence $\crow(4)=e_{1728}$.  The final assertions follow from
Proposition~\ref{prop:gross-orthogonality}.
\end{proof}

Corollary~\ref{cor:rational-relation} says that $(c_i\bmod2)_{i\in S_p}$ lies in $\Rp^\perp$.  Hence Theorem~\ref{thm:spanning} immediately implies Corollary~\ref{cor:parity}.  From this point on, the elliptic curve $E$ no longer enters the proof: the remaining task is the purely quaternionic statement that the vectors $\crow(D)$ span $\F_2^{S_p}$.

\section{Frobenius reduction and binary quadratic forms}\label{sec:binary}

Section~\ref{sec:gross} reduced the elliptic-curve application to the
parity of the integers $m_i(D)$. We now show that, at a rational
supersingular class, this parity is controlled by a binary rather than a
ternary quadratic form. We then show that distinct supersingular
coordinates give distinct inverse orbits and record the
genus-theoretic condition that will ensure admissibility of the
discriminants used later.

\subsection{The Frobenius-perpendicular lattice}

Fix $i\in S_p$ and choose an $\F_p$-model of $E_i$.  Its Frobenius $\pi_i\in R_i$ satisfies
\[
  \tr(\pi_i)=0,
  \qquad \pi_i^2=-p,
  \qquad \Nr(\pi_i)=p.
\]
Orthogonality in the trace-zero quaternion space will always refer to the
symmetric bilinear form obtained by polarizing the reduced norm:
\[
  (x,y)=\frac12\bigl(\Nr(x+y)-\Nr(x)-\Nr(y)\bigr)
       =-\frac12\tr(xy).
\]

\begin{definition*}
The \emph{Frobenius-perpendicular lattice} attached to $i$ is
\[
  L_i^\perp=L_i\cap(\Q\pi_i)^\perp.
\]
For $D>0$, write
\[
  N_i^\perp(D)=\#\{x\in L_i^\perp:\Nr(x)=D\}.
\]
\end{definition*}

For $j(E_i)\ne0,1728$, the two $\F_p$-models in the geometric class
$i$ are quadratic twists. After identifying them over
$\overline{\F}_p$, their Frobenius endomorphisms differ by sign. Hence
the line $\Q\pi_i$, and therefore $L_i^\perp$, depends only on the
nonexceptional geometric supersingular class $i$. For $j(E_i) \in \{0,1728\}$, we
retain the chosen $\F_p$-model; these coordinates are handled separately
by Lemma~\ref{lem:exceptional-rows}.

Kaneko's calculations with Ibukiyama's explicit maximal orders implicitly
produce the same binary norm forms, while He--Korpal--Tran--Vincent
explicitly construct the corresponding rank-two sublattices of the Gross
lattice; here we identify these sublattices intrinsically as the
Frobenius-perpendicular lattices $L_i^\perp$
\cite{Kaneko1989,Ibukiyama1982,HeKorpalTranVincent}.

Here the key point is the following parity reduction.

\begin{proposition}[Frobenius parity reduction]\label{prop:four-orbit}
If $p\nmid D$, then
\[
  m_i(D)\equiv \frac{N_i^\perp(D)}2\pmod2.
\]
\end{proposition}

\begin{proof}
Conjugation by Frobenius,
\[
  \Theta_i(x)=\pi_i x\pi_i^{-1},
\]
is the Galois action on geometric endomorphisms.  It preserves $R_i$, hence also $L_i$ and the reduced norm.  On the trace-zero quaternion space, its $+1$-eigenspace is the line $\Q\pi_i$ and its $-1$-eigenspace is $(\Q\pi_i)^\perp$; in particular, $\Theta_i(x)=-x$ is equivalent to $x\in(\Q\pi_i)^\perp$.

A norm-$D$ vector cannot lie on the Frobenius line: if $x=t\pi_i$ with $t\in\Q$, then $D=pt^2$, so
$v_p(D)=1+2v_p(t)$ is odd, contradicting $p\nmid D$.  If $x\notin L_i^\perp$, then $x$ lies in neither eigenspace, and
\[
  x,\quad -x,\quad \Theta_i(x),\quad -\Theta_i(x)
\]
are four distinct norm-$D$ vectors in $L_i$.  Hence
\[
  N_i(D)\equiv N_i^\perp(D)\pmod4.
\]
The involution $x\mapsto-x$ acts without fixed points on the norm-$D$ vectors in $L_i^\perp$, so $N_i^\perp(D)$ is even.  Lemma~\ref{lem:normalization} gives $N_i(D)=2m_i(D)$, and dividing the preceding congruence by two proves the result.
\end{proof}

Thus, modulo $2$, all norm-$D$ vectors outside the Frobenius anti-invariant plane cancel in four-element orbits.  The problem has been reduced to the binary lattice $L_i^\perp$.

\begin{remark}[Theta-series factorization modulo $4$]
Proposition~\ref{prop:four-orbit} is an instance of a general mod-$4$
factorization of theta series. Let $L$ be a positive-definite integral
lattice equipped with a norm-preserving involution $\sigma$, and put
\[
  L^{+}=\{x\in L:\sigma(x)=x\},
  \qquad
  L^{-}=\{x\in L:\sigma(x)=-x\}.
\]
Writing
\[
  \theta_M(q)=\sum_{x\in M}q^{\Nr(x)},
\]
one has the coefficientwise congruence
\[
  \theta_L(q)\equiv
  \theta_{L^{+}}(q)\theta_{L^{-}}(q)
  \pmod4.
\]
Indeed, grouping the vectors outside $L^{+}\cup L^{-}$ into
four-element orbits under $\sigma$ and negation gives
\[
  \theta_L\equiv\theta_{L^{+}}+\theta_{L^{-}}-1\pmod4.
\]
Moreover, every nonzero vector of $L^{+}$ and $L^{-}$ occurs together
with its negative, so every nonconstant coefficient of
$\theta_{L^{+}}-1$ and $\theta_{L^{-}}-1$ is even. Hence
\[
  (\theta_{L^{+}}-1)(\theta_{L^{-}}-1)\equiv0\pmod4,
\]
and therefore
\[
  \theta_L\equiv\theta_{L^{+}}\theta_{L^{-}}\pmod4.
\]

In the present setting, $\sigma$ is conjugation by Frobenius,
\[
  L_i^{+}=L_i\cap\Q\pi_i,
  \qquad
  L_i^{-}=L_i^\perp,
\]
and hence
\[
  \theta_{L_i}(q)\equiv
  \theta_{L_i^{+}}(q)\theta_{L_i^\perp}(q)
  \pmod4.
\]
\end{remark}

\subsection{Explicit Frobenius-compatible models}

To determine the norm form on $L_i^\perp$, it is not enough to know the
abstract maximal order $R_i$: one must also keep track of the Frobenius
line $\Q\pi_i$. The embedded quadratic order
\[
  \mathfrak o_i
  =R_i\cap\Q(\pi_i)
  =\End_{\F_p}(E_i)
\]
is either $\Z[\pi_i]$ or $\Z[(1+\pi_i)/2]$, and this determines which of
Ibukiyama's two explicit models applies.

Ibukiyama's models use an auxiliary prime $q\equiv3\pmod8$ with $\leg{p}{q}=-1$.  Write
\[
 B_{p,\infty}=\Q+\Q\alpha+\Q\beta+\Q\alpha\beta,
 \qquad
 \alpha^2=-p,\quad \beta^2=-q,\quad \alpha\beta=-\beta\alpha.
\]
If $r^2+p\equiv0\pmod q$, put
\[
 O(q,r)=\Z+\Z\frac{1+\beta}{2}
       +\Z\frac{\alpha(1+\beta)}2
       +\Z\frac{(r+\alpha)\beta}{q}.
\]
When $p\equiv3\pmod4$ and $(r')^2+p\equiv0\pmod{4q}$, put
\[
 O'(q,r')=\Z+\Z\frac{1+\alpha}{2}
       +\Z\beta
       +\Z\frac{(r'+\alpha)\beta}{2q}.
\]

\begin{proposition}[Ibukiyama models preserving Frobenius]
\label{prop:pair-model}
Assume $p>5$ and $j(E_i)\ne0,1728$. For suitable $q$ and $r$ or $r'$
as above, there is an inner automorphism of $B_{p,\infty}$ sending
\[
  (R_i,\pi_i)
\]
to
\[
  (O(q,r),\pm\alpha)
  \qquad\text{or}\qquad
  (O'(q,r'),\pm\alpha).
\]
The first case occurs when $\mathfrak o_i=\Z[\pi_i]$, and the second
when $\mathfrak o_i=\Z[(1+\pi_i)/2]$.
\end{proposition}

\begin{proof}
By Skolem--Noether we may first identify
$\Q(\pi_i)$ with $\Q(\alpha)$ so that $\pi_i$ is sent to
$\pm\alpha$. Under this identification the optimal quadratic order
\[
  \mathfrak o_i=R_i\cap\Q(\pi_i)
\]
becomes $\Z[\alpha]$ or $\Z[(1+\alpha)/2]$.

Ibukiyama's Theorems~1--2 give the corresponding standard maximal
orders, and the relative classification in
\cite[Section~4, Propositions~4.1--4.2]{Ibukiyama1982}
allows the isomorphism to be chosen preserving the specified optimal
quadratic order. Hence the distinguished Frobenius element is carried to
$\pm\alpha$. Since every $\Q$-algebra automorphism of
$B_{p,\infty}$ is inner, the assertion follows.
\end{proof}

The perpendicular norm forms can now be read off directly.

\begin{proposition}[Perpendicular norm forms]\label{prop:perp-forms}
Under the isomorphism of Proposition~\ref{prop:pair-model}, the lattice $L_i^\perp$ has the following descriptions.
\begin{enumerate}[label=\textup{(\alph*)}]
\item In $O(q,r)$, a basis is
\[
  \beta,\qquad \frac{2r\beta+2\alpha\beta}{q},
\]
and the norm form is
\[
  Q_{q,r}(x,y)=qx^2+4rxy+\frac{4(r^2+p)}q y^2.
\]
It is primitive of discriminant $-16p$.
\item In $O'(q,r')$, a basis is
\[
  2\beta,\qquad\frac{r'\beta+\alpha\beta}{q},
\]
and the norm form is $4Q'_{q,r'}$, where
\[
  Q'_{q,r'}(x,y)=qx^2+r'xy+\frac{(r')^2+p}{4q}y^2.
\]
The form $Q'_{q,r'}$ is primitive of discriminant $-p$.
\end{enumerate}
\end{proposition}

\begin{proof}
Consider first the order $O(q,r)$.  From its displayed basis, a general
trace-zero element of $\Z+2O(q,r)$ can be written as
\[
  b\beta+c(\alpha+\alpha\beta)
  +d\,\frac{2r\beta+2\alpha\beta}{q},
  \qquad b,c,d\in\Z.
\]
Since orthogonality to $\alpha$ is equivalent to vanishing of the
$\alpha$-component, we must have $c=0$.  Hence
\[
  L_i^\perp
  =\Z\beta+
   \Z\frac{2r\beta+2\alpha\beta}{q}.
\]

For $O'(q,r')$, a general trace-zero element of $\Z+2O'(q,r')$ has the form
\[
  b\alpha+2c\beta
  +d\,\frac{r'\beta+\alpha\beta}{q},
  \qquad b,c,d\in\Z.
\]
Orthogonality to $\alpha$ forces $b=0$, and therefore
\[
  L_i^\perp
  =\Z(2\beta)+
   \Z\frac{r'\beta+\alpha\beta}{q}.
\]

For $u,v\in\Q$, using
\[
  \alpha^2=-p,\qquad
  \beta^2=-q,\qquad
  \alpha\beta=-\beta\alpha,
\]
one obtains
\[
  \Nr((u+v\alpha)\beta)=q(u^2+pv^2).
\]
Substitution of the two displayed bases gives respectively
\[
  Q_{q,r}(x,y)
  =qx^2+4rxy+\frac{4(r^2+p)}q\,y^2
\]
and
\[
  4Q'_{q,r'}(x,y)
  =4\left(
      qx^2+r'xy+\frac{(r')^2+p}{4q}\,y^2
    \right).
\]
Their primitive parts have discriminants
\[
  (4r)^2-4q\frac{4(r^2+p)}q=-16p
\]
and
\[
  (r')^2-4q\frac{(r')^2+p}{4q}=-p.
\]
Finally, since $q$ is odd and $q\nmid r,r'$, the leading and middle
coefficients of each form are coprime. Hence both forms are primitive.
\end{proof}

The factor $4$ in the second case is important: every norm represented by $L_i^\perp$ in that branch is divisible by $4$.

\subsection{Matching form classes with supersingular classes}

The primitive forms obtained in Proposition~\ref{prop:perp-forms} are
exactly the forms associated by Xiao--Zhou--Deng--Qu with the corresponding
Ibukiyama orders; see
\cite[Section~3, especially Theorems~3.3 and~3.5]{XZDQ2025}.
We record the consequences of their parametrization that will be used
below.

For a nonexceptional $i\in S_p$, let $C_i$ denote the proper class of the
primitive form obtained from $L_i^\perp$. Reversing the orientation of the
binary lattice replaces a form $[a,b,c]$ by $[a,-b,c]$, and hence replaces
its proper class by its inverse. Thus the natural orientation-free object
attached to $i$ is the inverse orbit
\[
  \{C_i,C_i^{-1}\}.
\]

\begin{proposition}[Supersingular/form-class matching]
\label{prop:xzdq-compatibility}
Let $i\in S_p\setminus\{0,1728\}$. Under the
Xiao--Zhou--Deng--Qu parametrization, the two classes
\[
  C_i,\qquad C_i^{-1}
\]
correspond to the two $\F_p$-isomorphism classes in the quadratic-twist
pair with geometric supersingular class $i$. In particular, within either
branch the map
\[
  i\longmapsto\{C_i,C_i^{-1}\}
\]
is injective on the nonexceptional geometric supersingular classes, and
$C_i\ne C_i^{-1}$.

The form classes arising from the $O(q,r)$ branch are exactly the
nonexceptional classes in the nonprincipal genus of discriminant $-16p$.
The $O'(q,r')$ branch occurs only when $p\equiv3\pmod4$ and gives the
nonexceptional form classes of discriminant $-p$.
\end{proposition}

\begin{proof}
By Proposition~\ref{prop:pair-model}, the Frobenius-marked pair
$(R_i,\pi_i)$ is identified with the corresponding Ibukiyama model, and
Proposition~\ref{prop:perp-forms} shows that the primitive form obtained
from $L_i^\perp$ is the form occurring in
\cite[Theorems~3.3 and~3.5]{XZDQ2025}. Hence their parametrization
associates this form to the original geometric supersingular class $i$.

By \cite[Theorem~3.7]{XZDQ2025}, replacing a form class by its inverse
corresponds to passing to the quadratic twist over $\F_p$. Hence the
inverse orbit $\{C_i,C_i^{-1}\}$ corresponds precisely to the two
$\F_p$-isomorphism classes with geometric class $i$, and distinct
nonexceptional geometric supersingular classes give distinct inverse
orbits. Moreover, \cite[Remark~3.6]{XZDQ2025} shows that the only
self-inverse classes arising here are the exceptional classes
corresponding to $j=1728$; thus $C_i\ne C_i^{-1}$ for nonexceptional
$i$. The description of the two branches follows from
\cite[Theorems~3.5 and~3.7]{XZDQ2025}.
\end{proof}

Let $S_p^{(16)}$ denote the nonexceptional geometric supersingular classes
arising from the discriminant $-16p$ branch, and let $S_p^{(p)}$ denote
those arising from the discriminant $-p$ branch. Thus
\[
  S_p\setminus\{0,1728\}
  =
  S_p^{(16)}\sqcup S_p^{(p)},
\]
where absent exceptional classes are ignored, and
$S_p^{(p)}=\varnothing$ unless $p\equiv3\pmod4$.

For a primitive positive-definite binary form $C=[a,b,c]$, write
\[
  r_C(n)
  =
  \#\{(x,y)\in\Z^2:ax^2+bxy+cy^2=n\}.
\]
Since
\[
  r_C(n)=r_{C^{-1}}(n),
\]
the representation number depends only on the inverse orbit
$\{C,C^{-1}\}$. We may therefore translate the parity formula of
Proposition~\ref{prop:four-orbit} directly into the form-class language.

\begin{corollary}[Representation-parity dictionary]
\label{cor:dictionary}
Let $-D$ be admissible for $p$, and let
$i\in S_p\setminus\{0,1728\}$ correspond to the inverse orbit
$\{C_i,C_i^{-1}\}$.
\begin{enumerate}[label=\textup{(\alph*)}]
\item If $i\in S_p^{(16)}$, then
\[
  m_i(D)\equiv\frac{r_{C_i}(D)}2\pmod2.
\]

\item If $i\in S_p^{(p)}$ and $D=4n$, then
\[
  m_i(D)\equiv\frac{r_{C_i}(n)}2\pmod2.
\]

\item If $i\in S_p^{(p)}$ and $D$ is odd, then
\[
  m_i(D)\equiv0\pmod2.
\]
\end{enumerate}
\end{corollary}

\begin{proof}
Proposition~\ref{prop:four-orbit} reduces $m_i(D)$ modulo $2$ to half
the number of norm-$D$ vectors in $L_i^\perp$. In the $-16p$ branch,
Proposition~\ref{prop:perp-forms} identifies this norm form with a
representative of $C_i$, giving part~\textup{(a)}. In the $-p$ branch
the perpendicular norm form is $4Q'_{q,r'}$, which gives
part~\textup{(b)} and represents no odd integer, proving
part~\textup{(c)}.
\end{proof}

\subsection{The genus condition and admissibility}

We record the genus-theoretic fact needed later to verify admissibility in
the $-16p$ branch.

\begin{lemma}[Admissibility from the nonprincipal genus]
\label{lem:odd-genus}
Let $C$ lie in the nonprincipal genus of discriminant $-16p$. If $n$ is
squarefree, coprime to $2p$, and represented by $C$, then
\[
  n\equiv3\pmod4,
  \qquad
  \leg{-n}{p}=-1.
\]
In particular, $-n$ is admissible for $p$.
\end{lemma}

\begin{proof}
By genus theory, the nonprincipal genus of discriminant $-16p$ has genus
character
\[
  (-1)^{(n-1)/2}=-1
\]
on represented integers prime to $2p$; see
\cite[Theorem~3.15]{Cox2013}. Hence $n\equiv3\pmod4$.

Since $n$ is squarefree, its representation is primitive. By
\cite[Lemma~2.5]{Cox2013}, the discriminant $-16p$ is a quadratic residue
modulo $n$, and therefore
$\leg{-p}{n}=1$.
Quadratic reciprocity, together with $n\equiv3\pmod4$, gives
$\leg{-n}{p}=-1$.

Finally, $-n$ is a fundamental discriminant because $n$ is squarefree and
$n\equiv3\pmod4$.
\end{proof}

Corollary~\ref{cor:dictionary} translates the parity of the Gross
coefficients into representation parity by the corresponding binary form
classes, while Lemma~\ref{lem:odd-genus} ensures that the squarefree
integers represented in the discriminant $-16p$ branch give admissible
discriminants. The next section uses the form--ideal correspondence and
Chebotarev's theorem to control these representation parities.

\section{Form classes, ideals, and Chebotarev}\label{sec:classfield}

Corollary~\ref{cor:dictionary} expresses the parity of a Gross coefficient
in terms of representation by a binary form class. To control these
representations, we pass from form classes to the corresponding ideal classes
in quadratic orders, where class field theory and Chebotarev's theorem allow
us to choose split primes with prime ideals in specified ideal classes.
We collect here the standard facts needed for this purpose.

Let $\Delta<-4$ be a negative discriminant, put
$K=\Q(\sqrt{\Delta})$, and let
\[
  G_\Delta=\Pic(\OO_\Delta)
\]
be the proper ideal class group of the quadratic order $\OO_\Delta$ of
discriminant $\Delta$. Proper equivalence classes of primitive positive
definite binary quadratic forms of discriminant $\Delta$ are identified
with $G_\Delta$ via the usual form--ideal correspondence.
Let $H_\Delta$ be the ring class field of $\OO_\Delta$. Artin reciprocity
gives
\[
  \Gal(H_\Delta/K)\simeq G_\Delta,
\]
and $H_\Delta/\Q$ is generalized dihedral. If $\sigma_C$ denotes the
automorphism corresponding to $C\in G_\Delta$ and $c$ denotes complex
conjugation, then
\[
  c\sigma_Cc^{-1}=\sigma_C^{-1}.
\]
We also recall that the genera of discriminant $\Delta$ are the cosets of
$G_\Delta^2$ \cite[Theorem~3.15]{Cox2013}.

\begin{lemma}[Representations as ideal counts]
\label{lem:representations-ideals}
Let $C\in G_\Delta$ and $(n,\Delta)=1$. Then $r_C(n)/2$ is the number of
proper integral $\OO_\Delta$-ideals of norm $n$ in the ideal class
corresponding to $C$.
\end{lemma}

\begin{proof}
This is the classical form--ideal correspondence; see
\cite[Theorem~7.7 and Chapter~9]{Cox2013}. The two representations
$(x,y)$ and $(-x,-y)$ determine the same ideal, and for $\Delta<-4$ the
only proper automorphisms of a primitive positive-definite form are
$\pm I$.
\end{proof}

For a prime norm the preceding lemma is especially rigid: if
$\ell\nmid\Delta$ splits in $K$, there are exactly two proper ideals of
norm $\ell$, and their classes are inverse to one another. Chebotarev
therefore allows us to choose $\ell$ so that these two ideals represent
any prescribed inverse orbit.

\begin{lemma}[A prime in a prescribed inverse orbit]\label{lem:prime-isolation}
Let $C\in G_\Delta$ and let $\Sigma$ be a finite set of rational primes.  There are infinitely many primes $\ell\notin\Sigma$, unramified in $H_\Delta$, such that $\ell$ splits in $K$ and the two prime ideals above $\ell$ have classes $C$ and $C^{-1}$.  No form class outside $\{C,C^{-1}\}$ represents $\ell$.  Moreover,
\[
  \frac{r_C(\ell)}2=
  \begin{cases}
    1,&C\ne C^{-1},\\
    2,&C=C^{-1}.
  \end{cases}
\]
\end{lemma}

\begin{proof}
Via the Artin isomorphism, let $\sigma_C\in\Gal(H_\Delta/K)$ correspond to $C$.  Its conjugacy class in $\Gal(H_\Delta/\Q)$ is $\{\sigma_C,\sigma_C^{-1}\}$.  Chebotarev, with the primes in $\Sigma$ excluded, gives infinitely many rational primes with this Frobenius conjugacy class.  Its image in $\Gal(K/\Q)$ is trivial, so such a prime $\ell$ splits in $K$.  After choosing one prime $\mathfrak l$ above $\ell$, the two primes $\mathfrak l$ and $\overline{\mathfrak l}$ have classes $C$ and $C^{-1}$.  They are the only ideals of norm $\ell$, so the last assertions follow from Lemma~\ref{lem:representations-ideals}.
\end{proof}

For the discriminant $-p$ branch, which occurs only when
$p\equiv3\pmod4$, a refined Chebotarev argument produces a prime
$q\equiv1\pmod4$ that splits in $\Q(\sqrt{-p})$ and whose prime ideals
above $q$ represent $C_i$ and $C_i^{-1}$.

\begin{lemma}[Adding the condition $q\equiv1\pmod4$]\label{lem:class-mod-four}
Assume $p\equiv3\pmod4$, put $K=\Q(\sqrt{-p})$, and let $H$ be its Hilbert class field. For every $x\in\Pic(\OO_K)$ and every finite set $\Sigma$, there are infinitely many primes $q\notin\Sigma$ with $q\equiv1\pmod4$ such that one of the primes of $K$ above $q$ has class $x$.
\end{lemma}

\begin{proof}
The extension $H/K$ is unramified at every finite prime. Since $K\ne\Q(i)$ and the fundamental discriminant of $K$ is odd, the nontrivial quadratic extension $K(i)/K$ is ramified above $2$. If $H\cap K(i)$ were larger than $K$, then, because $K(i)/K$ is quadratic, one would have $K(i)\subseteq H$, contradicting the finite-prime unramifiedness of $H/K$. Hence
\[
  H\cap K(i)=K,
  \qquad
  \Gal(HK(i)/K)\simeq
  \Gal(H/K)\times\Gal(K(i)/K).
\]
Choose an element of this product whose restriction to $H$ corresponds under Artin reciprocity to $x$ and whose restriction to $K(i)$ is trivial. Its conjugacy class over $\Q$ consists of this element and the element with first component $x^{-1}$. Chebotarev gives infinitely many rational primes outside $\Sigma$ with this conjugacy class. Their image in $\Gal(K/\Q)$ is trivial, so they split in $K$; their Frobenius on $\Q(i)$ is also trivial, so $q\equiv1\pmod4$. Choosing one of the two primes above $q$, and conjugating it if necessary, gives ideal class $x$.
\end{proof}

Products of two split primes give a second, equally simple pattern.  Choose prime ideals $\mathfrak q_1\mid q_1$ and $\mathfrak q_2\mid q_2$ with classes $x$ and $y$.  The four choices of a prime above each $q_j$ give all ideals of norm $q_1q_2$.

\begin{lemma}[Ideals of semiprime norm]\label{lem:four-ideals}
Let $q_1\ne q_2$ be rational primes prime to $\Delta$ that split in $K=\Q(\sqrt\Delta)$, and choose proper invertible prime ideals above them with classes $x,y\in G_\Delta$. The four proper invertible ideals of norm $q_1q_2$ have classes
\[
  xy,\qquad xy^{-1},\qquad x^{-1}y,\qquad x^{-1}y^{-1}.
\]
If the inverse orbits of $xy$ and $xy^{-1}$ are distinct and non-self-inverse, each has odd representation parity and every other inverse orbit has even representation parity. If $xy$ is non-self-inverse and $xy^{-1}$ is self-inverse, only the inverse orbit of $xy$ has odd representation parity.
\end{lemma}

\begin{proof}
Write $q_j\OO_\Delta=\mathfrak q_j\overline{\mathfrak q}_j$ with $[\mathfrak q_1]=x$ and $[\mathfrak q_2]=y$. Since the primes are distinct and prime to $\Delta$, the four proper invertible ideals of norm $q_1q_2$ are obtained by choosing independently $\mathfrak q_j$ or $\overline{\mathfrak q}_j$; their classes are the four displayed products. Lemma~\ref{lem:representations-ideals} says that $r_C(q_1q_2)/2$ counts how many of these ideals lie in the class $C$. A non-self-inverse class occurring once therefore contributes odd parity, whereas a self-inverse class occurring twice contributes even parity.
\end{proof}

The genus condition needed for prescribed semiprime support is exactly what allows two inverse orbits to be realized by such a product.

\begin{lemma}[Two classes in the same genus]\label{lem:pair-realization}
If $A,B\in G_\Delta$ lie in the same genus, there are $x,y\in G_\Delta$ such that
\[
  xy=A,
  \qquad
  xy^{-1}=B.
\]
Moreover, $x$ and $y$ may be realized by chosen prime ideals above two distinct rational primes, avoiding any prescribed finite set.
\end{lemma}

\begin{proof}
Since $A$ and $B$ lie in the same genus, $AB^{-1}\in G_\Delta^2$.  Hence
\[
  AB=(AB^{-1})B^2
\]
is also a square.  Choose $x\in G_\Delta$ with $x^2=AB$ and put $y=Ax^{-1}$.  Then $xy=A$ and $xy^{-1}=B$.  Lemma~\ref{lem:prime-isolation}, applied twice and with an enlarged avoidance set after the first choice, realizes $x$ and $y$ by prime ideals above distinct rational primes.  If necessary, replace a chosen prime ideal by its conjugate to reverse the orientation.
\end{proof}

\section{Proof of the spanning theorem: isolating coordinates with primes}\label{sec:spanning-proof}

We prove Theorem~\ref{thm:spanning} by constructing, for each
$i\in S_p$, a Gross parity row equal to $e_i$ modulo coordinates already
known to lie in $\Rp$. We first treat the discriminant $-16p$ branch using
admissible discriminants $-\ell$ with $\ell$ prime; since $\ell$ is odd,
the discriminant $-p$ branch contributes zero. We then treat the
discriminant $-p$ branch using admissible discriminants of the form $-4q$.

Let
\[
W_{\mathrm{exc}}
=
\operatorname{Span}_{\mathbb F_2}
\{e_i : i\in \{0,1728\}\cap S_p\}.
\]
By Lemma~\ref{lem:exceptional-rows}, $W_{\mathrm{exc}}\subseteq\Rp$.  For $p=5,7,11$ one has respectively $S_5=\{0\}$, $S_7=\{1728\}$, and $S_{11}=\{0,1728\}$, so all rational supersingular classes are exceptional and there is nothing more to prove.  Assume henceforth that $p\ge13$.

We begin with $S_p^{(16)}$.  Fix $i\in S_p^{(16)}$ and choose a representative $C_i$ of its inverse orbit.  Proposition~\ref{prop:xzdq-compatibility} shows that $C_i$ is not self-inverse.  By Lemma~\ref{lem:prime-isolation}, we may choose a prime $\ell$, avoiding the primes dividing $6p$, such that the only form classes of discriminant $-16p$ representing $\ell$ are $C_i$ and $C_i^{-1}$.  Since these classes lie in the relevant nonprincipal genus, Lemma~\ref{lem:odd-genus} gives $\ell\equiv3\pmod4$ and $\leg{-\ell}{p}=-1$; hence $-\ell$ is admissible.

Corollary~\ref{cor:dictionary} now shows that the restriction of $\crow(\ell)$ to $S_p^{(16)}$ is $e_i$.  Its restriction to $S_p^{(p)}$ is zero because $\ell$ is odd.  Thus
\[
  \crow(\ell)\equiv e_i\pmod{W_{\mathrm{exc}}}.
\]
Both $\crow(\ell)$ and $W_{\mathrm{exc}}$ lie in $\Rp$, so $e_i\in\Rp$.  Varying $i$ gives
\[
  W_{16}:=W_{\mathrm{exc}}+
  \Span_{\F_2}\{e_i:i\in S_p^{(16)}\}
  \subseteq\Rp.
\]

It remains to treat $S_p^{(p)}$, which occurs only for $p\equiv3\pmod4$.  Fix $i\in S_p^{(p)}$ and choose a representative $C_i\in\Pic(\OO_{-p})$ of its inverse orbit.  By Proposition~\ref{prop:xzdq-compatibility}, $C_i$ is not self-inverse.  Lemma~\ref{lem:class-mod-four} gives a prime $q\equiv1\pmod4$, with $q\nmid2p$, such that a prime of $\Q(\sqrt{-p})$ above $q$ has class $C_i$; the conjugate prime has class $C_i^{-1}$.  Since these are the only ideals of norm $q$, Lemma~\ref{lem:representations-ideals} shows that $C_i^{\pm1}$ is the unique inverse orbit with odd representation parity at $q$.

The discriminant $-4q$ is fundamental. Moreover $q$ splits in $\Q(\sqrt{-p})$, so $\leg{-p}{q}=1$. Since $q\equiv1\pmod4$, this gives $\leg{q}{p}=1$ by quadratic reciprocity; because $p\equiv3\pmod4$, it follows that $\leg{-4q}{p}=-1$.  Hence $-4q$ is admissible.  By Corollary~\ref{cor:dictionary}, the restriction of $\crow(4q)$ to $S_p^{(p)}$ is $e_i$.  The row may also have entries on the already treated $-16p$ branch, but all of those coordinates lie in $W_{16}$.  Therefore
\[
  \crow(4q)\equiv e_i\pmod{W_{16}},
\]
and $e_i\in\Rp$.  Varying $i$ gives every remaining standard basis vector.  Consequently
\[
  \Rp=\F_2^{S_p},
\]
which proves Theorem~\ref{thm:spanning}.

\subsection{Semiprime rows}

Although products of two primes are not needed for the spanning theorem,
they reveal additional structure in the family of Gross rows: the same
class-group argument gives precise control over pairwise supports. We record
this stronger support phenomenon here.

\begin{proposition}[Semiprime support]
\label{prop:semiprime-support}
Let $i,j\in S_p^{(16)}$ be distinct. Then there are distinct odd primes
$q_1,q_2$ such that $-q_1q_2$ is admissible and
\[
  \crow(q_1q_2)\equiv e_i+e_j\pmod{W_{\mathrm{exc}}}.
\]
Thus every pair of nonexceptional coordinates in the discriminant $-16p$
branch occurs as the support, modulo the exceptional coordinates, of a
single admissible semiprime Gross row.
\end{proposition}

\begin{proof}
Choose representatives $C_i,C_j$ of the inverse orbits attached to
$i$ and $j$. Since they lie in the same nonprincipal genus of
discriminant $-16p$, Lemma~\ref{lem:pair-realization} gives classes
$x,y$ and distinct split primes $q_1,q_2$, coprime to $2p$, such that
\[
  xy=C_i,
  \qquad
  xy^{-1}=C_j,
\]
and chosen prime ideals above $q_1,q_2$ have classes $x,y$.  If
$n=q_1q_2$, the four ideals of norm $n$ have classes
\[
  C_i,\qquad C_j,\qquad C_j^{-1},\qquad C_i^{-1}.
\]
The two inverse orbits are distinct and non-self-inverse, so
Lemma~\ref{lem:four-ideals} and Corollary~\ref{cor:dictionary} give
\[
  \crow(n)\equiv e_i+e_j\pmod{W_{\mathrm{exc}}}.
\]
Moreover, $n$ is represented by the nonprincipal genus, and
Lemma~\ref{lem:odd-genus} shows that $-n$ is admissible. Since $n$ is
odd, the discriminant $-p$ branch contributes nothing.
\end{proof}

\section{Application to Watkins' conjecture}\label{sec:watkins}

Let $E/\Q$ be an elliptic curve of conductor $N$, and let $m_E$ denote
its modular degree, namely the minimal degree of a modular parametrization
\[
  X_0(N)\longrightarrow E.
\]
Watkins conjectured that
\[
  2^{\operatorname{rank}E(\Q)}\mid m_E
\]
\cite[Conjecture~4.1]{Watkins2002}.

The following argument is the prime-conductor argument of
\cite[Theorem~1.3]{KazalickiKohen2018}, with
Corollary~\ref{cor:parity} supplying the required parity of the Brandt
coefficients. We state explicitly the root-number $+1$ hypothesis used in
that argument.

\begin{theorem}\label{thm:modular-degree}
Let $E/\Q$ be an elliptic curve of prime conductor $p>3$, positive
Mordell--Weil rank, and root number $+1$. Then
\[
  4\mid m_E.
\]
In particular, Watkins' conjecture holds for such curves of rank $2$.
\end{theorem}

\begin{proof}
Let $E_0$ be the $\Gamma_0(p)$-optimal curve in the isogeny class of
$E$. The curves $E$ and $E_0$ have the same Brandt eigenvector, rank, and
root number. By Corollary~\ref{cor:parity}, its coefficients corresponding
to the classes in $S_p$ are even. Since the root number is $+1$, the
proof of \cite[Theorem~1.3]{KazalickiKohen2018} applies to $E_0$ and gives
\[
  4\mid m_{E_0}.
\]

It remains to pass from the optimal curve to the given curve $E$. Choose a modular parametrization
$\psi:X_0(p)\to E$ of degree $m_E$. After translating $\psi$ so that
$\psi(\infty)=0$, it factors through the optimal curve $E_0$; see
\cite[Remark~1.3]{CalegariEmerton2009}. Thus
\[
  \psi=\varphi\circ\pi_0
\]
for the optimal parametrization $\pi_0:X_0(p)\to E_0$ and a
$\Q$-isogeny $\varphi:E_0\to E$. Hence
\[
  \deg(\psi)=m_{E_0}\deg(\varphi),
\]
so $m_{E_0}\mid m_E$, and therefore $4\mid m_E$.
\end{proof}

\section{The case \texorpdfstring{$p=83$}{p=83}}\label{sec:example83}

The prime $83$ gives a small example in which both binary-form branches
occur and the constructions used in the proof of
Theorem~\ref{thm:spanning} can be seen explicitly. The rational
supersingular geometric classes are
\[
  S_{83}=\{0,17,28,50,67,68\},
  \qquad
  68\equiv1728\pmod{83}.
\]
The corresponding $\F_{83}$-isomorphism classes occur in quadratic-twist
pairs as in \cite[Example~3.8]{XZDQ2025}. Convenient reduced
representatives of the relevant inverse orbits are
\[
\begin{array}{ccl}
\toprule
\text{discriminant} & j & \text{form}\\
\midrule
-83   & 68 & [1,1,21]\\
      & 50 & [3,1,7]\\
\addlinespace
-1328 & 68 & [4,0,83]\\
      & 17 & [11,-6,31]\\
      & 28 & [7,4,48]\\
      & 67 & [16,-12,23]\\
      & 0  & [3,-2,111]\\
\bottomrule
\end{array}
\]
where inverse representatives are suppressed. Thus
\[
  S_{83}^{(16)}=\{17,28,67\},
  \qquad
  S_{83}^{(p)}=\{50\},
\]
while $0$ and $68$ are the exceptional coordinates supplied by $D=3$
and $D=4$.

The first stage of the spanning argument is already visible at $D=7$.
Among the nonexceptional classes of discriminant $-1328$, only
$[7,4,48]$ has odd half-representation number at $7$; indeed,
\[
  [7,4,48](\pm1,0)=7.
\]
Hence
\[
  \crow(7)\equiv e_{28}\pmod{W_{\mathrm{exc}}}.
\]

The second branch can be seen with
\[
  D=68=4\cdot17.
\]
At norm $68$ in the discriminant-$-1328$ branch, only the class
$[7,4,48]$ has odd half-representation number, while at norm $17$ in the
discriminant-$-83$ branch only $[3,1,7]$ does. For example,
\[
  [7,4,48](2,-1)=68,
  \qquad
  [3,1,7](2,-1)=17.
\]
Thus the nonexceptional support of $\crow(68)$ is
\[
  \{28,50\}.
\]
Since the coordinate $e_{28}$ has already been generated, this row gives
$e_{50}$ modulo the previously generated subspace. This illustrates the
two-stage structure of the proof of Theorem~\ref{thm:spanning}.

The additional semiprime relation recorded after the spanning theorem is
also visible with small examples. Representation-parity computations give
\[
\begin{array}{ccl}
\toprule
D & \text{factorization} &
\text{nonexceptional support on }S_{83}^{(16)}\\
\midrule
51  & 3\cdot17 & \{28,67\}\\
123 & 3\cdot41 & \{17,28\}\\
259 & 7\cdot37 & \{17,67\}\\
\bottomrule
\end{array}
\]
For instance,
\[
  [7,4,48](1,-1)=[16,-12,23](1,-1)=51,
\]
\[
  [11,-6,31](1,2)=[7,4,48](3,1)=123,
\]
and
\[
  [11,-6,31](3,-2)=[16,-12,23](1,-3)=259.
\]
No other nonexceptional inverse orbit of discriminant $-1328$ has
odd half-representation number at the corresponding value. Hence the
restrictions of these three admissible semiprime rows to $S_{83}^{(16)}$
are, respectively,
\[
  e_{28}+e_{67},\qquad
  e_{17}+e_{28},\qquad
  e_{17}+e_{67}.
\]
Thus, in this example, every pair of nonexceptional coordinates in the
discriminant $-16p$ branch is realized by a single semiprime Gross row, as
predicted by Proposition~\ref{prop:semiprime-support}.

\section*{Acknowledgments}

This research was supported by the European Union -- NextGenerationEU through the National Recovery and Resilience Plan 2021--2026 institutional grants of the University of Zagreb Faculty of Science (IK IA 1.1.3 Impact4Math, PMF-CROFUND). M.K. was also supported by the Croatian Science Foundation under the project IP-2022-10-5008 (TEBAG) and by the project ``Implementation of cutting-edge research and its application as part of the Scientific Center of Excellence for Quantum and Complex Systems, and Representations of Lie Algebras'', Grant No.~PK.1.1.10.0004, co-financed by the European Union through the European Regional Development Fund -- Competitiveness and Cohesion Programme 2021--2027.

During the development of this work, the authors used Aleph, a mathematical research assistant developed by Cantab Pi, and ChatGPT (OpenAI) for literature searches, exploratory work on proof strategies, and assistance with drafting and editorial review. The authors independently checked all mathematical statements, proofs, computations, and citations and take full responsibility for the contents of the paper.

\end{document}